\documentclass[12pt]{amsart}
\usepackage[dvips]{graphicx}
\usepackage{amssymb,amsmath,amsthm}

\newtheorem{Theorem}{Theorem}[section]
\newtheorem{Lemma}[Theorem]{Lemma}

\newtheorem{Proposition}[Theorem]{Proposition}
\theoremstyle{definition}

\newtheorem{Conjecture}[Theorem]{Conjecture}

\def\RR{{\mathbb R}}
\def\ZZ{{\mathbb Z}}

\def\CC{{\mathbb C}}

\begin{document}

\title{Roots of the Ehrhart polynomial of hypersimplices}
\author{Hidefumi Ohsugi and Kazuki Shibata}

\address{Hidefumi Ohsugi,
Department of Mathematics,
College of Science,
Rikkyo University,
Toshima-ku, Tokyo 171-8501, Japan.} 
\email{ohsugi@rikkyo.ac.jp}

\address{Kazuki Shibata,
Department of Mathematics,
Graduate School of Science,
Rikkyo University,
Toshima-ku, Tokyo 171-8501, Japan.}
\email{12rc003c@rikkyo.ac.jp}

\date{}

\begin{abstract}
The Ehrhart polynomial of the $d$-th hypersimplex $\Delta(d,n)$ of order $n$ is studied.
By computational experiments and a known result for $d=2$,
we conjecture that the real part of 
every roots of the Ehrhart polynomial of $\Delta(d,n)$
is negative and larger than $- \frac{n}{d}$ if $n \geq 2d$.
In this paper, we show that the conjecture is true when $d=3$ and that
every root $a$ of the Ehrhart polynomial of $\Delta(d,n)$
satisfies
$
-\frac{n}{d} < {\rm Re} (a) < 1
$
if $4 \leq d  \ll n$.
\end{abstract}

\maketitle

\section*{Introduction}

Let ${\mathcal P} \subset \RR^n$ be an integral convex polytope
of dimension $p$.
Recall that an {\em integral} convex polytope is 
a convex polytope all of whose vertices have integer coordinates.
Given an integer $m >0$, we write $i({\mathcal P}, m)$ for the number of
integer points belonging to
$m {\mathcal P} = \{ m \alpha \ | \ \alpha \in {\mathcal P}\}$, that is,
$$
i({\mathcal P}, m) = | m {\mathcal P} \cap \ZZ^n |
\ \ \ \ m= 1,2,\ldots.
$$
It is known that $i({\mathcal P}, m)$ is a polynomial in $m$ of 
degree $p$.
We call $i({\mathcal P}, m)$ the {\em Ehrhart polynomial} of ${\mathcal P}$.
In general, $i({\mathcal P}, 0) = 1$ and the leading coefficient of 
$i({\mathcal P}, m)$ is equal to the normalized volume of ${\mathcal P}$.
In \cite{Stanley}, it was conjectured that
each root $ a \in \CC$ of $i({\mathcal P}, m)$ satisfies 
$- p \leq {\rm Re} (a) \leq p-1$.
However, several counterexamples for  the conjecture are given in 
\cite{Hig, OhSh} recently.
On the other hand, it is known \cite{Bra} that
all the roots of $i({\mathcal P}, m)$ lie inside the disc with center $-\frac{1}{2}$
and radius $p(p-\frac{1}{2})$.

In this paper, we study roots of the Ehrhart polynomial of a hypersimplex.
Let $d$ and $n$ be integers such that $1 \leq d < n$.
The {\em hypersimplex} $\Delta(d,n)$ is a convex polytope in $\RR^n$
which is the convex hull of
$$
\{
{\bf e}_{i_1} + \cdots + {\bf e}_{i_d}
\ | \ 
1 \leq i_1 < \cdots < i_d \leq n
\},
$$
where each ${\bf e}_j$ is the unit coordinate vector of $\RR^n$.
In general, it is known that
\begin{itemize}
\item
The dimension of $\Delta(d,n)$ is $n-1$;
\item
$\Delta(d,n)$ is isomorphic to $\Delta(n-d,n)$.
\end{itemize}
Thus, throughout this paper, we always assume that $d$ and $n$ satisfy the condition
\begin{eqnarray}
2 d \leq n.
\end{eqnarray}
The Ehrhart polynomial $i(\Delta(d,n), m)$ of $\Delta(d,n)$ is given in \cite{Kat}:
$$
i(\Delta(d,n), m) = \sum_{s=0}^{d-1}(-1)^s
         {n \choose s}
         { (d-s) m+n-1 -s
         \choose
          n-1
          }.
$$
Katzman computed the Hilbert polynomial of corresponding semigroup rings and 
it is equal to $i(\Delta(d,n), m)$ since the semigroup ring is normal.
Exactly speaking, the Hilbert polynomial is equal to the {\em normalized} Ehrhart polynomial
if and only if the semigroup ring is normal.
For the sake of completeness, we will later show that  the normalized Ehrhart polynomial
is equal to  the Ehrhart polynomial in this case.

If  $d=1$, then  $i(\Delta(1,n), m)$ is an $(n-1)$-simplex
and
$$
i(\Delta(1,n), m) = 
         { m+n-1
         \choose
          n-1
          }.
$$
Hence, roots of $i(\Delta(1,n), m) $ are
$ -(n-1), -(n-2), \ldots, -2, -1$.
If $d =2$, then
$$
i(\Delta(2,n), m) =
         { 2 m+n-1
         \choose
          n-1
          }
-
n
         { m+n-2
         \choose
          n-1
          }.
$$
In \cite{five}, it is shown that every root $a \in \CC$ of $i(\Delta(2,n), m)$
satisfies
$$
-\frac{n}{2} < {\rm Re} (a) < 0
$$
when $2d = 4 \leq n$.
Computational experiments\footnote{A rough bound was obtained by
Masanori Tajima in his master's thesis (in Japanese).}
suggest the following conjecture:

\begin{Conjecture}
\label{conjecture}
Let $2d \leq n$.
Then, 
every root $a \in \CC$ of $i(\Delta(d,n), m)$
satisfies
$$
-\frac{n}{d} < {\rm Re} (a) < 0.
$$
\end{Conjecture}

In this paper, we show that 

\begin{Theorem}
Let $d$ and $n$ be positive integers, and
let $a \in \CC$ be a root of  $i(\Delta(d,n), m)$.
Then, we have the following:
\begin{enumerate}
\item[(i)]
If $d=3$ and $n \geq 6$, then we have
$
-\frac{n}{3} < {\rm Re} (a) < 0.
$
\item[(ii)]
If $4 \leq d  \ll n$, then we have
$
-\frac{n}{d} < {\rm Re} (a) < 1.
$
\end{enumerate}
\end{Theorem}


\section{Fundamental facts on $i(\Delta(d,n), m)$}

In this section, we present fundamental facts on $i(\Delta(d,n), m)$
which will later play an important role.
First we confirm that
the Ehrhart polynomial of  the hypersimplex 
$\Delta(d,n)$ is
\begin{eqnarray}
\label{ehrhart}
i(\Delta(d,n),m) = \sum_{s=0}^{d-1}(-1)^s
         {n \choose s}
         { (d-s) m+n-1 -s
         \choose
          n-1
          }.
\end{eqnarray}
It is pointed out in \cite[Remark 2.3]{Kat} that the right-hand side of (\ref{ehrhart})
is equal to the {\it normalized} Ehrhart polynomial of $\Delta(d,n)$.
In other words,
$$
| m  \Delta(d,n) \cap \ZZ A |
 = \sum_{s=0}^{d-1}(-1)^s
         {n \choose s}
         { (d-s) m+n-1 -s
         \choose
          n-1
          },
$$
where $A =\{
{\bf e}_{i_1} + \cdots + {\bf e}_{i_d}
\ | \ 
1 \leq i_1 < \cdots < i_d \leq n
\}$.
Note that $\ZZ A \neq \ZZ^n$.
Since the following fact is not stated in \cite{Kat},
we show it for the sake of completeness:

\begin{Proposition}
The Ehrhart polynomial of $\Delta(d,n)$
is equal to 
the normalized Ehrhart polynomial of $\Delta(d,n)$.
\end{Proposition}

\begin{proof}
In general, we have $m  \Delta(d,n) \cap \ZZ A \subset m  \Delta(d,n) \cap \ZZ^n$.
Hence, it is enough to show that 
$m  \Delta(d,n) \cap \ZZ A \supset m  \Delta(d,n) \cap \ZZ^n$.
Let $\alpha =( \alpha_1,\ldots, \alpha_n) \in m  \Delta(d,n) \cap \ZZ^n $.
Since $\alpha \in m  \Delta(d,n) $, 
we have $\alpha_1+ \cdots + \alpha_n = d m$.
Remark that 
$${\bf e}_1 - {\bf e}_2 = ({\bf e}_1 + {\bf e}_3 + \cdots + {\bf e}_{d+1})
- ({\bf e}_2 + {\bf e}_3 + \cdots + {\bf e}_{d+1}) \in \ZZ A.$$
Similarly, 
${\bf e}_1 - {\bf e}_j $ belongs to $\ZZ A$ for each $2 \leq j \leq n$.
Hence,
$$
d {\bf e}_1 =  ({\bf e}_1 + {\bf e}_2 + \cdots + {\bf e}_d)
+ \sum_{j=2}^d ({\bf e}_1 - {\bf e}_j ) \in \ZZ A.
$$
Thus,
$$
\alpha = 
\left(
\sum_{j=1}^n \alpha_j
\right)
 {\bf e}_1 -
\sum_{j=2}^n \alpha_j ({\bf e}_1 - {\bf e}_j )
=
d m \ 
 {\bf e}_1 -
\sum_{j=2}^n \alpha_j ({\bf e}_1 - {\bf e}_j )
$$
belongs to $\ZZ A$,
as desired.
\end{proof}

In order to study roots of $ i(\Delta(d,n),m)$, we will use the following 
fact:

\begin{Proposition}[Rouch\'e]
Let $D$ be a simply connected region
and let $f$ and $g$ be holomorphic functions in $\overline{D}$.
If $|f(z)| > |g (z)|$ holds for every $z \in \partial D $, then
$f$ and $f + g$ have the same number of zeros in $D$,
where each zero is counted as many times as its multiplicity.
\end{Proposition}

Fix an integer $d >0$.
For $s= 0,1,\ldots,d-1$,
let
$$
f_{n,s} (m)  =
{n \choose s}  
((d-s) m+n-1 -s) \cdots ((d-s) m+ 1 -s).
$$
Then
$$
i(\Delta(d,n),m) =
\frac{1}{(n-1)!}
 \sum_{s=0}^{d-1}(-1)^s f_{n,s} (m).
$$
We will apply Rouch\'e Theorem by considering
the functions
$
f(z) = f_{n,0} (z)
$
and
$$
g(z) = \sum_{s=1}^{d-1}(-1)^s f_{n,s} (z).
$$
Let 
$
\varphi_{n,d,s} (z)
=
\frac{
|f_{n,s} (z)|
}
{
|f_{n,0} (z)|
}
$.

\begin{Lemma}
\label{migi}
For every $z = \beta \sqrt{-1}$ with $\beta \in \RR$, 
we have 
$
\varphi_{n+1,d,s} (z)
<
\varphi_{n,d,s} (z)
$.
\end{Lemma}

\begin{proof}
Since
$$
\varphi_{n,d,s} (\beta \sqrt{-1})\\
=
{n \choose s}  
\sqrt{
\frac{
((d-s)^2 \beta^2 + (n-1 -s)^2) \cdots ((d-s)^2 \beta^2 + (1 -s)^2) 
}
{
(d^2 \beta^2 + (n-1)^2) \cdots (d^2 \beta^2 + 1)
}
}
$$
holds, we have
\begin{eqnarray*}
\frac{\varphi_{n+1,d,s} (z)}{\varphi_{n,d,s} (z)}
&=&
\sqrt{
\frac{(n+1)^2 (d-s)^2 \beta^2 + (n+1)^2(n -s)^2 }{(n+1-s)^2 d^2 \beta^2 + (n+1-s)^2 n^2}
}.
\end{eqnarray*}
Moreover, since
\begin{eqnarray*}
 & &
((n+1-s)^2 d^2 \beta^2 + (n+1-s)^2 n^2)
-
((n+1)^2 (d-s)^2 \beta^2 + (n+1)^2(n -s)^2)\\
& = &
\beta^2 ((n+1) (d-s) + (n+1-s) d) (n-d+1) s
+
(  (n+1) (n -s)  + (n+1-s)  n) s\\
&>& 0
\end{eqnarray*}
we have 
$
\frac{\varphi_{n+1,d,s} (z)}{\varphi_{n,d,s} (z)}
<1
$, as desired.
\end{proof}

\begin{Lemma}
\label{hidari}
Suppose $n \geq d^2 -2$.
Then, for every $z =  - \frac{n}{d} - \beta \sqrt{-1}$ with $\beta \in \RR$, 
we have 
$
\varphi_{n+d,d,s} (z)
<
\varphi_{n,d,s} (z)
$.
\end{Lemma}

\begin{proof}
Since $\varphi_{n,d,s} (z)$ is equal to
$$
{n \choose s}  
\sqrt{
\frac{
((d-s)^2 \beta^2 + (-1 -s+\frac{ s n}{d})^2) \cdots ((d-s)^2 \beta^2 + (-n+1 -s+\frac{s n}{d})^2)
}
{
(d^2 \beta^2 + 1) \cdots (d^2 \beta^2 + (n-1)^2)
}
},
$$
we have
\begin{eqnarray*}
& &
\frac{\varphi_{n+d,d,s} (z)}{\varphi_{n,d,s} (z)}\\
&=&
\frac{{n+d \choose s}  }{ {n \choose s}  }
\sqrt{
\frac{
((d-s)^2 \beta^2 + (1 -\frac{ s n}{d})^2) \cdots ((d-s)^2 \beta^2 + (s-\frac{s n}{d})^2)
}
{
(d^2 \beta^2 + n^2) \cdots (d^2 \beta^2 + (n+s-1)^2)
}
}
\\
& &
\times 
\sqrt{
\frac{
((d-s)^2 \beta^2 + (n+s -\frac{ s n}{d})^2) \cdots ((d-s)^2 \beta^2 + (n+d-1-\frac{s n}{d})^2)
}
{
(d^2 \beta^2 + (n+s)^2) \cdots (d^2 \beta^2 + (n+d-1)^2)
}
}
\\
& < &
\frac{
(n+d)(n+d-1) \cdots (n+d+1-s)
}{
n(n-1) \cdots (n+1-s)
}
\\
& &
\times
\sqrt{
\frac{
((d-s)^2 \beta^2 + (1 -\frac{ s n}{d})^2) \cdots ((d-s)^2 \beta^2 + (s-\frac{s n}{d})^2)
}
{
(d^2 \beta^2 + n^2) \cdots (d^2 \beta^2 + (n+s-1)^2)
}
}.
\end{eqnarray*}
For $1 \leq k \leq s$ $(\leq d-1)$,
\begin{eqnarray*}
\frac{
((d-s)^2 \beta^2 + (k -\frac{ s n}{d})^2) 
}{
(d^2 \beta^2 + (n+k-1)^2)
}
&=&
\left(
\frac{d-1}{d}
\right)^2
\frac{
((\frac{d-s}{d-1})^2 \beta^2 + (\frac{k -\frac{ s n}{d}}{d-1})^2) 
}{
(\beta^2 + (\frac{n+k-1}{d})^2)
}
\end{eqnarray*}
and
\begin{eqnarray*}
\frac{|n+k-1|}{d}- \frac{|k -\frac{ s n}{d}|}{d-1}
&=&
\frac{(n+k-1)(d-1) +dk - sn}{d(d-1)}\\
&=&
\frac{n(d - 1 - s) + dk +(d- 1)(k-1)}{d(d-1)}
> 0.
\end{eqnarray*}
Hence, we have
\begin{eqnarray*}
\sqrt{
\frac{
(d-s)^2 \beta^2 + (k -\frac{ s n}{d})^2
}{
d^2 \beta^2 + (n+1-k)^2
}
}
& < &
\frac{d-1}{d}.
\end{eqnarray*}
Thus,
\begin{eqnarray*}
\frac{\varphi_{n+d,d,s} (z)}{\varphi_{n,d,s} (z)}
& < &
\frac{
(n+d)(n+d-1) \cdots (n+d+1-s)
}{
n(n-1) \cdots (n+1-s)
}
\left(
\frac{d-1}{d}
\right)^s.
\end{eqnarray*}
Moreover, for $1 \leq k \leq s$ $(\leq d-1)$,
$$
(n+1-k) d - (n+d+1-k) (d-1)
=
n-(d^2-2)+(d-1-k) \geq 0
$$
and hence
$$
\frac{n+d+1-k}{n+1-k}
\cdot
\frac{d-1}{d}
\leq 1.
$$
Therefore, we have
$ 
\frac{\varphi_{n+d,d,s} (z)}{\varphi_{n,d,s} (z)} < 1
$, as desired.
\end{proof}

\begin{Lemma}
\label{aida}
For every $z =   - \alpha + \lambda n \sqrt{-1}$ with $0 \leq \alpha \leq \frac{n}{d}$ and $\lambda \in \RR$, 
we have 
$$\varphi_{n,d,s} (z) <  {n \choose s} \left( \frac{(d-s)^2 + \frac{1}{\lambda^2} }{d^2} \right)^\frac{n-1}{2}.$$
\end{Lemma}

\begin{proof}
Note that
$$
\varphi_{n,d,s} (z) 
=
{n \choose s}  
\sqrt{
\prod_{k=1}^{n-1}
\frac{
(d-s)^2 \lambda^2 n^2 + (k -s- (d-s) \alpha )^2
}
{
d^2\lambda^2 n^2 + (k - d \alpha )^2 
}
}.
$$
For $1 \leq k \leq n-1$ and $1 \leq s \leq d-1$, we have
$$
-n
<
- n +s \left(\frac{n}{d}-1\right)+1 =
1-s - (d-s) \frac{n}{d}
<
k -s- (d-s) \alpha
< n-1 -s < n.
$$
Hence,
$(k -s- (d-s) \alpha )^2 < n^2$.
Thus,
$$
\varphi_{n,d,s} (z)
< 
{n \choose s}  
\left(
\frac{
(d-s)^2 \lambda^2 n^2 + n^2
}
{
d^2\lambda^2 n^2
}
\right)^{\frac{n-1}{2}}
=
 {n \choose s} \left( \frac{(d-s)^2 + \frac{1}{\lambda^2} }{d^2} \right)^\frac{n-1}{2},
$$
as desired.
\end{proof}
\section{The case of $d=3$}

In this section, we prove that Conjecture \ref{conjecture} is true if $d =3$.

\begin{Theorem}
Let $d=3$ and $n \geq 6 \ (=2d)$.
Then, 
every root $a \in \CC$ of $i(\Delta(3,n), m)$
satisfies
$$
-\frac{n}{3} < {\rm Re} (a) < 0.
$$
\end{Theorem}

\begin{proof}
If $n=6$, then the Ehrhart polynomial of $\Delta(3,6)$
is
\begin{eqnarray*}
i(\Delta(3,6), z) &=& 
         { 3 z+5
         \choose
          5
          }
-
6         { 2z+4
         \choose
          5
          }
+
15
         {  z+3
         \choose
         5
          }\\
&=&
\frac{1}{20}
(z+1) \left( 11 (z+1)^4 + 5 (z+1)^2 + 4            \right) .
\end{eqnarray*}
Since $0 < 4 < 5 < 11$ holds,
by Enestr\"om--Kakeya Theorem (see, e.g., \cite{EK}),
it follows that
every root $a \in \CC$ satisfies $| (a + 1)^2 | < 1$.
Hence, in particular, we have $-2< {\rm Re} (a) < 0 $.

Let $n \geq 7$.
We apply Rouch\'e Theorem for
the functions
$$
f(z) = f_{n,0} (z), \ 
g(z) = - f_{n,1} (z) + f_{n,2} (z) 
$$ and 
the region
$$
D=
\left\{
z \in \CC \ \left| \ 
-\frac{n}{3} < {\rm Re} (z) < 0,\ 
- \sqrt{2} n  < {\rm Im} (z) < \sqrt{2} n 
\right.
\right\}.
$$
Remark that the roots of $f(z)$ are
$$
-\frac{n-1}{3}, -\frac{n-2}{3},\ldots,-\frac{2}{3},-\frac{1}{3}
$$
and all of them belong to $D$.
Thus, it is enough to show that
$$
 | f_{n,1} (z)| + | f_{n,2} (z)| < |f_{n,0} (z)|
$$
for all $z \in \partial D$.

\bigskip

\noindent
{\bf Case 1.}
$z = \beta \sqrt{-1}$ where $\beta \in \RR$.

If $n=7$ and $s=1$, then
\begin{eqnarray*}
\frac{
|f_{7,1} (\beta \sqrt{-1})|
}
{
|f_{7,0} (\beta \sqrt{-1})|
}
&  =&
7
\sqrt{
\frac{
(4 \beta^2 + 5^2) (4 \beta^2 + 4^2) (4 \beta^2 + 3^2) (4 \beta^2 + 2^2) (4 \beta^2 + 1^2)  (4 \beta^2)
}
{
(9 \beta^2 + 6^2) (9 \beta^2 + 5^2) (9 \beta^2 + 4^2)(9 \beta^2 + 3^2)(9 \beta^2 + 2^2) (9 \beta^2 + 1)
}
}\\
&  =&
\frac{1792}{2187}
\sqrt{
\frac{
9
(\beta^2 + \frac{25}{4})  (\beta^2 + \frac{9}{4}) (\beta^2 + \frac{1}{4}) \beta^2
}
{
16
 (\beta^2 + \frac{25}{9}) (\beta^2 + \frac{16}{9})(\beta^2 + \frac{4}{9}) (\beta^2 + \frac{1}{9})
}
}.
\end{eqnarray*}
We now show
$$
\frac{
9
(\beta^2 + \frac{25}{4})  (\beta^2 + \frac{9}{4}) (\beta^2 + \frac{1}{4}) \beta^2
}
{
16
 (\beta^2 + \frac{25}{9}) (\beta^2 + \frac{16}{9})(\beta^2 + \frac{4}{9}) (\beta^2 + \frac{1}{9})
 }
<1.
$$
Let
\begin{eqnarray*}
f(x) &=&
16
 \left(x + \frac{25}{9}\right) \left(x + \frac{16}{9}\right) \left(x + \frac{4}{9}\right) \left(x + \frac{1}{9}\right)\\
& & 
-
9
\left(x + \frac{25}{4}\right)  \left(x + \frac{9}{4}\right) \left(x + \frac{1}{4}\right) x\\
&=&
7x^4 + \frac{109}{36} x^3 - \frac{10969}{432} x^2 + \frac{739711}{46656} x+\frac{25600}{6561}.
\end{eqnarray*}
Then, since
$$
f(y+1) = 7y^4 + \frac{1117}{36} y^3 + \frac{11099}{432} y^2 + \frac{100567}{46656} y+\frac{1844635}{419904}
>
0$$
for all $y \geq 0$,
it follows that $f(x) >0$ for all $x \geq 1$.
Moreover, if $0 \leq x < 1$, then
\begin{eqnarray*}
f(x) &>&
6x^4 + 3 x^3 - 27 x^2 + 15 x+3\\
&=&
3 (1-x) (1+x+x (2x+5)(1- x))\\
&>&
0.
\end{eqnarray*}
Thus, $f(x) > 0$ for all $x \geq 0$. 

If $n=7$ and $s=2$, 
\begin{eqnarray*}
& &
\frac{
|f_{7,2} (\beta \sqrt{-1})|
}
{
|f_{7,0} (\beta \sqrt{-1})|
}\\
&  =&
21
\sqrt{
\frac{
(\beta^2 + 4^2) (\beta^2 + 3^2) (\beta^2 + 2^2) (\beta^2 + 1^2) (\beta^2 + 0^2)  (\beta^2 + (-1)^2)
}
{
(9 \beta^2 + 6^2) (9 \beta^2 + 5^2) (9 \beta^2 + 4^2)(9 \beta^2 + 3^2)(9 \beta^2 + 2^2) (9 \beta^2 + 1)
}
}\\
&  =&
\frac{14}{81}
\sqrt{
\frac{
(\beta^2 + 16) (\beta^2 + 9)(\beta^2 + 1) \beta^2  
}
{
36 (\beta^2 + \frac{25}{9}) (\beta^2 + \frac{16}{9})(\beta^2 + \frac{4}{9}) (\beta^2 + \frac{1}{9})
}
}.
\end{eqnarray*}
It then follows that
$$
\frac{
(\beta^2 + 16) (\beta^2 + 9)(\beta^2 + 1) \beta^2  
}
{
36 (\beta^2 + \frac{25}{9}) (\beta^2 + \frac{16}{9})(\beta^2 + \frac{4}{9}) (\beta^2 + \frac{1}{9})
}
<1
$$
since
\begin{eqnarray*}
& &
36
 \left(x + \frac{25}{9}\right) \left(x + \frac{16}{9}\right) \left(x + \frac{4}{9}\right) \left(x + \frac{1}{9}\right)
-
\left(x + 16 \right)  \left(x +  9 \right) \left(x +1 \right) x\\
&=&
35 x^4 + 158 x^3 + \frac{5}{3} x^2 +  \frac{232}{81} x+\frac{3484}{729}
+(10 x -2)^2\\
&>&0
\end{eqnarray*}
for all $x \geq 0$.

Thus, by Lemma \ref{migi},
if $n \geq 7$,
then
\begin{eqnarray*}
\frac{
|f_{n,1}(\beta \sqrt{-1})| + |f_{n,2}(\beta \sqrt{-1})|
}
{
|f_{n,0}(\beta \sqrt{-1})|
}
& \leq &
\frac{
|f_{7,1}(\beta \sqrt{-1})|
}
{
|f_{7,0}(\beta \sqrt{-1})|
}
+
\frac{
|f_{7,2}(\beta \sqrt{-1})|
}
{
|f_{7,0}(\beta \sqrt{-1})|
}\\
& < &
\frac{1792}{2187}+\frac{14}{81}\\
&<& 1.
\end{eqnarray*}
Hence, we have
$
|f_{n,1}(\beta \sqrt{-1})| + |f_{n,2}(\beta \sqrt{-1})|
<
|f_{n,0}(\beta \sqrt{-1})|.
$

\bigskip

\noindent
{\bf Case 2.} 
$z= -\frac{n}{3} + \beta \sqrt{-1} $ with $\beta \in \RR$.

First, we study the case when $n=7, 8, 9$.
If $n=7$, then
\begin{eqnarray*}
& &
\frac{
|f_{7,1} (z)|
}
{
|f_{7,0} (z)|
}\\
&  =&
7
\sqrt{
\frac{
(4 \beta^2 + (\frac{1}{3})^2) 
(4 \beta^2 + (\frac{2}{3})^2) 
(4 \beta^2 + (\frac{5}{3})^2) 
(4 \beta^2 + (\frac{8}{3})^2)
(4 \beta^2 + (\frac{11}{3})^2)
(4 \beta^2 + (\frac{14}{3})^2)
}
{
(9 \beta^2 + 1)
(9 \beta^2 + 2^2)
(9 \beta^2 + 3^2)
(9 \beta^2 + 4^2)
(9 \beta^2 + 5^2)
(9 \beta^2 + 6^2)}
}\\
&  =&
7
\left( \frac{2}{3} \right)^4
\sqrt{
\frac{
(4 \beta^2 + (\frac{1}{3})^2) 
(4 \beta^2 + (\frac{2}{3})^2) 
(4 \beta^2 + (\frac{5}{3})^2) 
(4 \beta^2 + (\frac{8}{3})^2)
}
{
(4 \beta^2 + (\frac{2}{3})^2)
(4 \beta^2 + (\frac{2}{3} \cdot 2)^2)
(4 \beta^2 + (\frac{2}{3} \cdot 3)^2)
(4 \beta^2 +(\frac{2}{3} \cdot 4)^2)
}
}\\
& &
\times
\frac{ \frac{11}{3} }{5} \cdot\frac{ \frac{14}{3} }{6}
\sqrt{
\frac{
((\frac{6}{11})^2 \beta^2 + 1)  ((\frac{3}{7})^2 \beta^2 +1)
}
{
((\frac{3}{5})^2 \beta^2 + 1)  ((\frac{1}{2})^2 \beta^2 + 1)
}
}\\
&\leq&
\frac{8624}{10935},
\end{eqnarray*}
and
\begin{eqnarray*}
& &
\frac{
|f_{7,2} (z)|
}
{
|f_{7,0} (z)|
}\\
&  =&
21
\sqrt{
\frac{
(\beta^2 + (\frac{5}{3})^2) 
(\beta^2 + (\frac{2}{3})^2) 
(\beta^2 + (\frac{1}{3})^2) 
(\beta^2 + (\frac{4}{3})^2) 
(\beta^2 + (\frac{7}{3})^2) 
 (\beta^2 + (\frac{10}{3})^2)
}
{
(9 \beta^2 + 1)
(9 \beta^2 + 2^2)
(9 \beta^2 + 3^2)
(9 \beta^2 + 4^2)
(9 \beta^2 + 5^2)
(9 \beta^2 + 6^2)
}
}\\
&  =&
21 
\cdot \frac{1}{2}
\cdot \frac{1}{3^6}
\cdot \frac{7}{3}
\cdot \frac{10}{3}
\sqrt{
\frac{
( (\frac{3}{7})^2 \beta^2 + 1) 
 ((\frac{3}{10})^2 \beta^2 + 1)
}
{
(\beta^2 + 1)
((\frac{1}{2})^2\beta^2 + 1)
}
}\\
&\leq &
\frac{245}{2187}.
\end{eqnarray*}
Then, 
$
\frac{8624}{10935}
+
\frac{245}{2187}
<1.
$
Moreover, if $n=8$, then
$$
\frac{
|f_{8,1} (z)|
}
{
|f_{8,0} (z)|
}
<
\frac{8}{7}
\sqrt{
\frac{
4 \beta^2 + (\frac{16}{3})^2
}
{
9 \beta^2 + 7^2
}
}
\frac{
|f_{7,1} (z)|
}
{
|f_{7,0} (z)|
}
<
\frac{
|f_{7,1} (z)|
}
{
|f_{7,0} (z)|
},
$$
and
$$
\frac{
|f_{8,2} (z)|
}
{
|f_{8,0} (z)|
}
 < 
\frac{4}{3}
\sqrt{
\frac{
\beta^2 + (\frac{11}{3})^2
}
{
9 \beta^2 +7^2
}
}
\frac{
|f_{7,2} (z)|
}
{
|f_{7,0} (z)|
}
 < 
\frac{
|f_{7,2} (z)|
}
{
|f_{7,0} (z)|
}.
$$
On the other hand, if $n=9$, then we have
\begin{eqnarray*}
& &
\frac{
|f_{9,1} (z)|
}
{
|f_{9,0} (z)|
}\\
&  =&
9 \cdot \frac{2^6}{3^6}
\sqrt{
\frac{
9 \beta^2
(9 \beta^2 + \left(\frac{3}{2}\right)^2)
(9 \beta^2 + \left(\frac{3}{2}\right)^2)
(9 \beta^2 + \left(\frac{9}{2}\right)^2)
(4 \beta^2 + 5^2)
(4 \beta^2 + 6^2)
}
{
(9 \beta^2 + 1)
(9 \beta^2 + 2^2)
(9 \beta^2 + 4^2)
(9 \beta^2 + 5^2)
 (9 \beta^2 +7^2)
 (9 \beta^2 + 8^2)
}
}\\
&<&
\frac{64}{81},
\end{eqnarray*}
and
\begin{eqnarray*}
& &
\frac{
|f_{9,2} (z)|
}
{
|f_{9,0} (z)|
}\\
&  =&
36
\cdot
\frac{1}{2^4 \cdot 3^3}
\sqrt{
\frac{
9 \beta^2 
(4 \beta^2 + 2^2)
(4\beta^2 + 4^2)
(\beta^2 + 3^2)
(4\beta^2 + 6^2)
(4\beta^2 + 8^2)
}
{
(9 \beta^2 + 1) 
(9 \beta^2 + 2^2)
(9 \beta^2 + 4^2)
(9 \beta^2 + 5^2)
 (9 \beta^2 +7^2)
 (9 \beta^2 + 8^2)
}
}\\
&<& \frac{1}{12}.
\end{eqnarray*}
Then, 
$
\frac{64}{81}
+
\frac{1}{12}
<1
.$

Note that $d^2 -2 = 7 \leq n$.
By Lemma \ref{hidari}, it follows that
$$
 | f_{n,1} (z)| + | f_{n,2} (z)| < |f_{n,0} (z)|
.$$

\bigskip

\noindent
{\bf Case 3.}
$z = -\alpha \pm \sqrt{2} n \sqrt{-1}$ with $0 \leq \alpha \leq \frac{n}{3}$.

By Lemma \ref{aida},
$$
\frac{
|f_{n,1} (z)|
}
{
|f_{n,0} (z)|
} <  n \left( \frac{4 + \frac{1}{2} }{9} \right)^\frac{n-1}{2}
=
n \left(\frac{1}{2} \right)^{\frac{n-1}{2}},
$$
and
$$
\frac{
|f_{n,2} (z)|
}
{
|f_{n,0} (z)|
}
< \frac{n(n-1)}{2} \left( \frac{1 + \frac{1}{2} }{9} \right)^\frac{n-1}{2}
=
\frac{n(n-1)}{2} 
\left(\frac{1}{6} \right)^{\frac{n-1}{2}}.
$$
Since
$$
\frac{
(n+1)  \left(\frac{1}{2} \right)^{\frac{n}{2}}
}
{
n \left(\frac{1}{2} \right)^{\frac{n-1}{2}}
}
=
\frac{n+1}{\sqrt{2} n} < 1
$$
and
$$
\frac{
\frac{n(n+1)}{2} 
\left(\frac{1}{6} \right)^{\frac{n}{2}}
}
{
\frac{n(n-1)}{2} 
\left(\frac{1}{6} \right)^{\frac{n-1}{2}}
}
=\frac{n+1}{\sqrt{6} (n-1)} < 1
$$
hold for $n \geq 7$, we have
$$
\frac{
|f_{n,1} (z)|
}
{
|f_{n,0} (z)|
} < 
n \left(\frac{1}{2} \right)^{\frac{n-1}{2}}
\leq 7 \left(\frac{1}{2} \right)^{3}=\frac{7}{8}
$$
and
$$
\frac{
|f_{n,2} (z)|
}
{
|f_{n,0} (z)|
}
< 
\frac{n(n-1)}{2} 
\left(\frac{1}{6} \right)^{\frac{n-1}{2}}
<
21\left(\frac{1}{6} \right)^3=\frac{7}{72}.
$$
Thus,
$
|f_{n,1}(z)| + |f_{n,2}(z)|
<
|f_{n,0}(z)|
$
follows from
$
\frac{7}{8} +\frac{7}{72}
< 1.
$
\end{proof}
\section{The case of $d\geq 4$}

In this section, we study the case of $d \geq 4$.
Although, we could not prove that 
Conjecture \ref{conjecture}
is true for $d \geq 4$,
we prove inequalities which are close to those in 
Conjecture \ref{conjecture} when $d \ll n$.
%

\begin{Theorem}
Suppose that integers $d$ and $n$ satisfy
$d \geq 4$ and $n \geq 6d^2-16d+13$.
Then, 
every root $a \in \CC$ of $i(\Delta(d,n), m)$
satisfies
$$
{\rm Re} (a) < 1.
$$
\end{Theorem}

\begin{proof}
First, we prove that,
every $z \in \CC$ with
$1 \leq {\rm Re} (z) $
satisfies
$$
\frac{ |  f_{n,s} (z)  |}{ | f_{n,0} (z) |}
<
 \frac{2^s }{3^s \ s!}
$$
for $s = 1,2, \ldots, d-1$.
Let $z = \alpha +1 + \beta \sqrt{-1}$ with $\alpha \geq 0$ and
$\beta \in \RR$
and let $m = z-1 \in \CC$.
Then, 
$$
f_{n,s} (z)  =
{n \choose s}  
((d-s) m+d+n-1 -2s) \cdots ((d-s) m+d+ 1 -2s).
$$
For $i = 1,2,\ldots,n-d$, we have
$$
d(d+n-i-2s)- 
(d-s)(d+n-i)
=
s(n-d-i) \geq 0,
$$
and hence,
$$
0<
\frac{d-s}{d+n-i-2s} 
\leq \frac{d}{d+n-i}.
$$
Thus,
\begin{eqnarray*}
\left|
\frac{
(d-s) m+d+n-i -2s
}
{
d m+d+n-i
}
\right|
&=&
\frac{d+n-i -2s}{d+n-i}
\sqrt{
\frac{
(\frac{d-s}{d+n-i -2s} \beta)^2+ (\frac{d-s}{d+n-i -2s} \alpha +1)^2
}
{
(\frac{d}{d+n-i} \beta)^2+ (\frac{d}{d+n-i} \alpha +1)^2
}
}\\
& \leq &
\frac{d+n-i -2s}{d+n-i}.
\end{eqnarray*}
On the other hand,
for $j = 1,2,\ldots, s$, since
$$
(d-1)(2d-j) -d(2d-j-2s) = 
2d (s-1) + j  >0
$$
and
$$
(d-1)(2d-j) +d(2d-j-2s) = 
2d(2(d - 1 - s) + (s - j) + 1) +j
>0
$$
hold, we have
$$
\left|\frac{2d-j-2s}{d-1} \right|
< 
\frac{2d-j}{d}.
$$
Thus,
\begin{eqnarray*}
\left|
\frac{
(d-s) m+2d -j -2s
}
{
d m+2d -j
}
\right|
&=&
\frac{d-1}{d}
\left|
\frac{
\frac{d-s}{d-1} m+\frac{2d -j -2s}{d-1}
}
{
m+\frac{2d -j}{d}
}
\right|\\
&=&
\frac{d-1}{d}
\sqrt{
\frac{
(\frac{d-s}{d -1} \beta)^2+ (\frac{d-s}{d-1} \alpha +\frac{2d -j -2s}{d-1})^2
}
{
\beta^2+ (\alpha +\frac{2d-j}{d} )^2
}
}\\
& < &
\frac{d-1}{d}.
\end{eqnarray*}
In addition, for $k=1,2,\ldots,s$,
\begin{eqnarray*}
& & (2d-1)(2d-2)
(d+n-k)
-
(2d-2k+1)(2d-2k)(d+n-1)\\
&=&
2(k-1)
(
2(d+n)(d-1-k) + 2dn+n+2k
) \geq 0.
\end{eqnarray*}
Hence,
$$
\frac{(2d-2k+1)(2d-2k)}{d+n-k}
\leq 
\frac{(2d-1)(2d-2)}{d+n-1}.
$$
Therefore,
\begin{eqnarray*}
& &
\frac{|f_{n,s} (z) |}{|f_{n,0} (z) |} \\
&=&
{n \choose s}  
\frac{
|(d-s) m+d+n-1 -2s| \cdots |(d-s) m+d+ 1 -2s|
}
{
|d m+d+n-1| \cdots |d m+d+ 1|
}\\
&=&
{n \choose s}  
\prod_{i=1}^{n-d}
\left|
\frac{
(d-s) m+d+n-i -2s
}
{
d m+d+n-i 
}
\right|
\ 
\prod_{j=1}^{s}
\left|
\frac{
(d-s) m+2d-j -2s
}
{
d m+2d-j 
}
\right|\\
& &
\times 
\prod_{j=s+1}^{d-1}
\left|
\frac{
(d-s) m+2d-j -2s
}
{
d m+2d-j 
}
\right|\\
&<&
{n \choose s}  
\prod_{i=1}^{n-d}
\frac{
d+n-i -2s
}
{
d+n-i 
}
\ 
\left(
\frac{d-1}{d}
\right)^{s}
\\
& = &
{n \choose s}  
\frac{
(2d-1) \cdots (2d-2s)
}
{
(d+n-1) \cdots (d+n-2s)
}
\left(
\frac{d-1}{d}
\right)^{s}\\
&=&
\frac{1}{s!}
\left(
\prod_{k=1}^s
\frac{n+1-k}{d+n-k-s}
\right)
\left(
\prod_{k=1}^s
\frac{(2d-2k+1)(2d-2k)}{d+n-k}
\right)
\left(
\frac{d-1}{d}
\right)^{s}\\
& \leq &
\frac{1}{s!}
\left(
\frac{(2d-1)(2d-2)}{d+n-1}
\cdot
\frac{d-1}{d}
\right)^{s}\\
& \leq &
\frac{2^s}{3^s s!}
\left(
\frac{(2d-1)(d-1)^2}{d(  2d^2-5d+4)}
\right)^{s}.
\end{eqnarray*}
Since $d\geq 4$,
$$
1-\frac{(2d-1)(d-1)^2}{d(  2d^2-5d+4)}
=\frac{1}{d(  2d^2-5d+4)}
=\frac{1}{d( 2 (d-1)(d-2) +d)}
>0.
$$
Thus,
$$
\sum_{s=1}^{d-1} 
\frac{
 |  f_{n,s} (z)  |}{ | f_{n,0} (z) |}
<
\sum_{s=1}^{d-1} \frac{2^s}{3^s s!}
<
-1+e^{\frac{2}{3}}
<
1,
$$
and hence
$$ \left|\sum_{s=1}^{d-1} (-1)^s f_{n,s} (z)  \right| <| f_{n,0} (z) |.$$
Therefore,
$z$ is not a root of $i(\Delta(d,n),m)$.
\end{proof}

\begin{Theorem}
Suppose that integers $d$ and $n$ satisfy
$d \geq 4$ and $n \geq d^2 + 2 d$.
Then, 
every root $a \in \CC$ of $i(\Delta(d,n), m)$
satisfies
$$
- \frac{n}{d} < {\rm Re} (a) .
$$
\end{Theorem}

\begin{proof}
We prove that,
every $z \in \CC$ with
${\rm Re} (z) \leq - \frac{n}{d}$
satisfies
$$
\frac{ |  f_{n,s} (z)  |}{ | f_{n,0} (z) |}
<
 \frac{1}{d-1}
$$
for $s = 1,2, \ldots, d-1$.
Let $z = -\alpha -\frac{n}{d} - \beta \sqrt{-1}$ with $\alpha \geq 0$ and
$\beta \in \RR$
and let $m = -z-\frac{n}{d} =\alpha + \beta \sqrt{-1}\in \CC$.
Then,
\begin{eqnarray*}
f_{n,s} (z)  
 &=&
(-1)^{n-1}
{n \choose s}  
((d-s) m- \frac{ns}{d} +1+s) \cdots ((d-s) m-\frac{ns}{d} +n-1 +s)
\end{eqnarray*}
and
\begin{eqnarray*}
& &
\frac{|f_{n,s} (z)  |}{|f_{n,0} (z) |} \\
&=&
{n \choose s}  
\frac{
|(d-s) m- \frac{ns}{d} +1+s | \cdots | (d-s) m-\frac{ns}{d} +n-1 +s|
}
{
|d m+1| \cdots |d m+n-1|
}\\
&=&
{n \choose s}  
\prod_{k=1}^{\lfloor \frac{ns}{d} \rfloor -s}
\left|
\frac{
(d-s) m-( \frac{ns}{d} - \lfloor \frac{ns}{d} \rfloor -1+k)
}
{
d m+k
}
\right|
\\
& &
\times
\prod_{k=\lfloor \frac{ns}{d} \rfloor -s+1}^{n-d}
\left|
\frac{
 (d-s) m-\frac{ns}{d} +k+s
}
{
d m+k
}
\right|
\times
\prod_{k=n-d+1}^{n-1}
\left|
\frac{
 (d-s) m-\frac{ns}{d} +k+s
}
{
d m+k
}
\right|.
\end{eqnarray*}

For 
$k=1,2,\ldots,\left\lfloor \frac{ns}{d} \right\rfloor -s$,
we have
$$
-k <
1-\left( \frac{ns}{d} - \left\lfloor \frac{ns}{d} \right\rfloor \right)-k
\leq
1-k \ (\leq 0).$$
Hence,
$$
\left|
\frac{
(d-s) m-( \frac{ns}{d} - \left\lfloor \frac{ns}{d} \right\rfloor -1+k)
}
{
d m+k
}
\right|
<1.
$$

For 
$k=\left\lfloor \frac{ns}{d} \right\rfloor -s+1,
\left\lfloor \frac{ns}{d} \right\rfloor -s+2,\ldots,n-d$,
we have
$$
-\frac{ns}{d} +k+s >0
$$
and
$$
\frac{k}{d}
-
\frac{-\frac{ns}{d} +k+s}{d-s}
=
\frac{s}{d(d-s)}
(n-d-k) \geq 0.
$$
Hence,
for 
$k=\left\lfloor \frac{ns}{d} \right\rfloor -s+1,
\left\lfloor \frac{ns}{d} \right\rfloor -s+2,\ldots,n-d$,
$$
\left|
\frac{
 (d-s) m-\frac{ns}{d} +k+s
}
{
d m+k
}
\right|
=
\frac{d-s}{d}
\left|
\frac{
m+\frac{-\frac{ns}{d} +k+s}{d-s}
}
{
m+\frac{k}{d}
}
\right|
< 
\frac{d-s}{d}.
$$

For
$k=n-d+1,n-d+2,\ldots,n-1$, we have
$$
\frac{d}{k}-\frac{d-s}{-\frac{ns}{d} +k+s}
=
\frac{d s (k-(n-d))}{k (kd -(n-d)s )}>0.
$$
Hence
$$
\left|
\frac{
 (d-s) m-\frac{ns}{d} +k+s
}
{
d m+k
}
\right|
=
\frac{-\frac{ns}{d} +k+s}{k}
\left|
\frac{
\frac{d-s}{-\frac{ns}{d} +k+s}
m+1
}
{
\frac{d}{k} m+1
}
\right|
<
\frac{-\frac{ns}{d} +k+s}{k}.
$$

Therefore,
\begin{eqnarray*}
& &
\frac{|f_{n,s} \left(-m-\frac{n}{d}\right)  |}{|f_{n,0}\left(-m-\frac{n}{d}\right) |} \\
& \leq &
{n \choose s}  
\left(
\frac{d-s}{d}
\right)^{n-d- \lfloor \frac{ns}{d} \rfloor+s}
\frac{
-\frac{ns}{d} +n-d+s+1
}
{
n-d+1
}
\cdots
\frac{
-\frac{ns}{d} +n-1 +s
}
{
n-1
}\\
& \leq &
\frac{n^s}{s!}  
\left(
\frac{d-s}{d}
\right)^{(d-s)(\frac{n}{d} -1)}
\frac{
-\frac{ns}{d} +n-d+s+1
}
{
n-d+1
}
\cdots
\frac{
-\frac{ns}{d} +n-1 +s
}
{
n-1
}.
\end{eqnarray*}
Let
$$
g (n,d,s) = \log
\left(
\frac{n^s}{\Gamma(s+1)}  
\left(
\frac{d-s}{d}
\right)^{(d-s)(\frac{n}{d} -1)}
\right),
$$
where $\Gamma (-)$ is the gamma function.
Since $1 \leq  s \leq d-1$,
we have 
$$
\log \frac{d-s}{d} < -\frac{s}{d}
$$
and hence,
for any $n \geq d^2+2 d$,
\begin{eqnarray*}
\frac{\partial g}{\partial n}
&=&
\frac{s}{n}
+
\frac{d-s}{d}
\log
\left(\frac{d-s}{d}\right)\\
&<&
\frac{s}{d^2+2d}
+
\frac{d-s}{d}
\left(-\frac{s}{d} \right)\\
&=&
\frac{s}{d^2(d+2)}(-2-(d+2)(d-1-s))\\
&<&0.
\end{eqnarray*}
Thus, $g(n+1,d,s) < g(n,d,s)$.
Moreover,
for $k=1,2,\ldots,d-1$ ($<n$),
$$
\frac{\partial}{\partial n}
\left(
\frac{
-\frac{ns}{d} +n-k+s
}
{
n-k
}
\right)
=
-
\frac{
s (d-k)
}{d (n-k)^2}
< 0.
$$
Therefore, for $n \geq d^2+2d$,
\begin{eqnarray*}
& &
\frac{|f_{n,s} \left(-m-\frac{n}{d}\right)  |}{|f_{n,0}\left(-m-\frac{n}{d}\right) |} \\
& \leq &
\frac{(d^2+2d)^s}{s!}  
\left(
\frac{d-s}{d}
\right)^{(d-s)(d+1)}
\frac{
(d-s)(d+1)+1
}
{
d^2+d+1
}
\cdots
\frac{
(d-s)(d+1)+d-1
}
{
d^2+2d-1
}.
\end{eqnarray*}

\noindent
{\bf Case 1.}
Suppose $s=d-1$.

For each $d\geq 4$,
\begin{eqnarray*}
& &
\frac{(d^2+2d)^{d-1}}{(d-1)!}  
\left(
\frac{1}{d}
\right)^{d+1}
\frac{
d+2
}
{
d^2+d+1
}
\cdots
\frac{
2d
}
{
d^2+2d-1
}\\
&=&
\frac{1}{d}
\frac{(d+2)^{d-1}}{d!}  
\frac{
d+2
}
{
d^2+d+1
}
\cdots
\frac{
2d
}
{
d^2+2d-1
}\\
&=&
\frac{1}{d}
\cdot
\frac{2^{d-2} \ 3}{d!}
\frac{
2d(d+2)
}
{
3(d^2+2d-1)
}
\prod_{k=1}^{d-2}
\frac{
(d+2)(d+1+k)
}
{
2(d^2+d+k)
}.
\end{eqnarray*}
Note that
$$
\frac{2^{d-2} \ 3}{d!}
 = \prod_{k=4}^d \frac{2}{k} < 1,
$$
$$
3(d^2+2d-1)-2d(d+2)
=(d-1)(d+3) > 0, 
$$
and, for each $1\leq k \leq d-2$
$$
2(d^2+d+k)- (d+2)(d+1+k)
=
d(d  - 2-k) + d -2 >0.
$$
Hence, 
$$
\frac{(d^2+2d)^{d-1}}{(d-1)!}  
\left(
\frac{1}{d}
\right)^{d+1}
\frac{
d+2
}
{
d^2+d+1
}
\cdots
\frac{
2d
}
{
d^2+2d-1
}
<
\frac{1}{d-1}.
$$

\noindent
{\bf
Case 2.}
Suppose $1 \leq s \leq d-2$.

Let
$$h(d,s) = 
\log
\left(
(d-1)
\frac{(d^2+2d)^s}{\Gamma(s+1)}
\left(
\frac{d-s}{d}
\right)^{(d-s)(d+1)}
\right).
$$
Then, for $1 \leq s \leq d-2$,
$$ 
\frac{\partial h}{\partial s}
=
\log (d^2+2d)
-
\frac{\Gamma(s+1)'}{\Gamma(s+1)}
-(d+1)
\left(
\log \frac{d-s}{d}
+
1
\right)
$$
and
$$
\frac{\partial^2 h}{\partial s^2 }
=
\frac{
d+1}
{
d-s
}
-\sum_{\ell=0}^\infty \frac{1}{(s+1+\ell)^2}
\geq
\frac{
d+1}
{
d-s
}
-
\left(
\frac{\pi^2}{6}
-1
\right)
=
\frac{
s+1}
{
d-s
}
+
\left(
2-
\frac{\pi^2}{6}
\right)
>0.
$$
Thus, for each $d$, 
we have $h(d,s) \leq \max (h(d,1), h(d,d-2) ) $
for all $1 \leq s \leq d-2$.

We now show that $h(d,1) <0$ and $h(d,d-2) <0$.
Note that, for $d \geq 4$,
$$
h(d,1) 
= 
\log
(d-1)
+\log
(d^2+2d)
+
(d-1)(d+1)
\log
\left(
\frac{d-1}{d}
\right)
$$
and
\begin{eqnarray*}
\frac{\partial h(d,1)}{\partial d}
&=&
\frac{1}{d-1}
+\frac{2d+2}{d^2+2d}
+
\frac{d+1}{d}
+
2d
\log
\left(
\frac{d-1}{d}
\right)
\\
&<&
\frac{1}{d-1}
+\frac{2d+2}{d^2+2d}
+
\frac{d+1}{d}
-
2
\\
&=&
-\frac{(d-4)(d^2+d-1)}{(d-1)d(d+2)}\\
&\leq&0.
\end{eqnarray*}
Thus, for every $d \geq 4$,
we have
$h (d,1) \leq h(4,1) = \log 72+15 \log \frac{3}{4} 
= \log \frac{129140163}{134217728} <0$.

On the other hand,
\begin{eqnarray*}
(d-1)
\frac{(d^2+2d)^{d-2}}{(d-2)!}
\left(
\frac{2}{d}
\right)^{2 (d+1)}
&=& 
\frac{2^6\ \ 6^{d-2} (d-1)}{ (d-2)! d^6}
\left(
\frac{4 (d^2+2d)}{6 d^2}
\right)^{d-2}\\
&=&
\frac{2^6\ \ 6^{d-2} (d-1)}{(d-2)! d^6}
\left(
\frac{3d -(d-4)}{3 d}
\right)^{d-2}\\
&\leq &
\frac{2^6\ \ 6^{d-2} (d-1)}{(d-2)! d^6}.
\\
& = &
\frac{2^6\ \ 6^{d-2} (d-1)^2 (d+1)}{(d+1)! d^5}
\\
& = &
\frac{2^6\ \ 6^{d-2} (d-1)}{(d+1)! d^3}
\cdot
\frac{d^2-1}{d^2}\\
& < &
\frac{2^6\ \ 6^{d-2} (d-1)}{(d+1)! d^3}.
\end{eqnarray*}
For $d = 4$,
$$
\frac{2^6\ \ 6^{4-2} (4-1)}{(4+1)!\ 4^3} = \frac{9}{10} < 1,
$$
and for $d \geq 5$,
$$
\frac{2^6\ \ 6^{d-2} (d-1)}{(d+1)! d^3}
=
\frac{2^6\ \ 6^2 }{5!\ 5^2}
\cdot
\frac{5^2}{d^2}
\cdot
\frac{d-1}{d}
\prod_{k=6}^{d+1}
\frac{6}{k}
<
\frac{96}{125}
<1.
$$
Thus,
\begin{eqnarray*}
h(d,d-2)
&=& 
\log
\left(
(d-1)
\frac{(d^2+2d)^{d-2}}{(d-2)!}
\left(
\frac{2}{d}
\right)^{2 (d+1)}
\right) <0.
\end{eqnarray*}
Therefore, 
it follows that, for every 
$1 \leq s \leq d-1$,
\begin{eqnarray*}
& &
\frac{|f_{n,s} \left(-m-\frac{n}{d}\right)  |}{|f_{n,0}\left(-m-\frac{n}{d}\right) |} \\
& \leq &
\frac{(d^2+2d)^s}{s!}  
\left(
\frac{d-s}{d}
\right)^{(d-s)(d+1)}
\frac{
(d-s)(d+1)+1
}
{
d^2+d+1
}
\cdots
\frac{
(d-s)(d+1)+d-1
}
{
d^2+2d-1
}\\
&<&
\frac{1}{d-1}.
\end{eqnarray*}

Hence,
$$
\sum_{s=1}^{d-1} 
\frac{
 |  f_{n,s} (z)  |}{ | f_{n,0} (z) |}
<
\sum_{s=1}^{d-1} \frac{1}{d-1}
=
1,
$$
and
$$ \left|\sum_{s=1}^{d-1} (-1)^s f_{n,s} (z)  \right| <| f_{n,0} (z) |.$$
Thus,
$z$ is not a root of $i(\Delta(d,n),m)$.
\end{proof}
\section{Computational experiments}

In this section, some computational experiments are given.
First, we computed the approximate roots of the Ehrhart polynomial
$i(\Delta(d,n),m)$
\begin{itemize}
\item
for $4 \le d \le 10$ and $2d \le n \le d^2 + 2d$ \ \ \ (Figures \ref{d4} -- \ref{d10}),
\item
for $4 \le d \le 75$ and $n=2d$  \ \ \  (Figure \ref{d4_75}),
\end{itemize}
by using the software package \texttt{Mathematica}
\cite{Mathematica}~(``\texttt{N}'' and ``\texttt{Solve}") and \texttt{gnuplot}.
Second,
by the software \texttt{Maple} \cite{Maple}~(``\texttt{Hurwitz}"),
we checked that every root $\alpha$ of the Ehrhart polynomial $i(\Delta(d,n),m)$ satisfies
$$
- \frac{n}{d} < {\rm Re} (\alpha) < 0
$$
for $4 \le d \le 10$ and $2d \le n \le d^2 + 2d$.


\begin{figure}[htbp]
\begin{center}
\begin{tabular}{c}
\begin{minipage}{0.5\hsize}
\begin{center}
\rotatebox{-90}{\includegraphics[width=5cm]{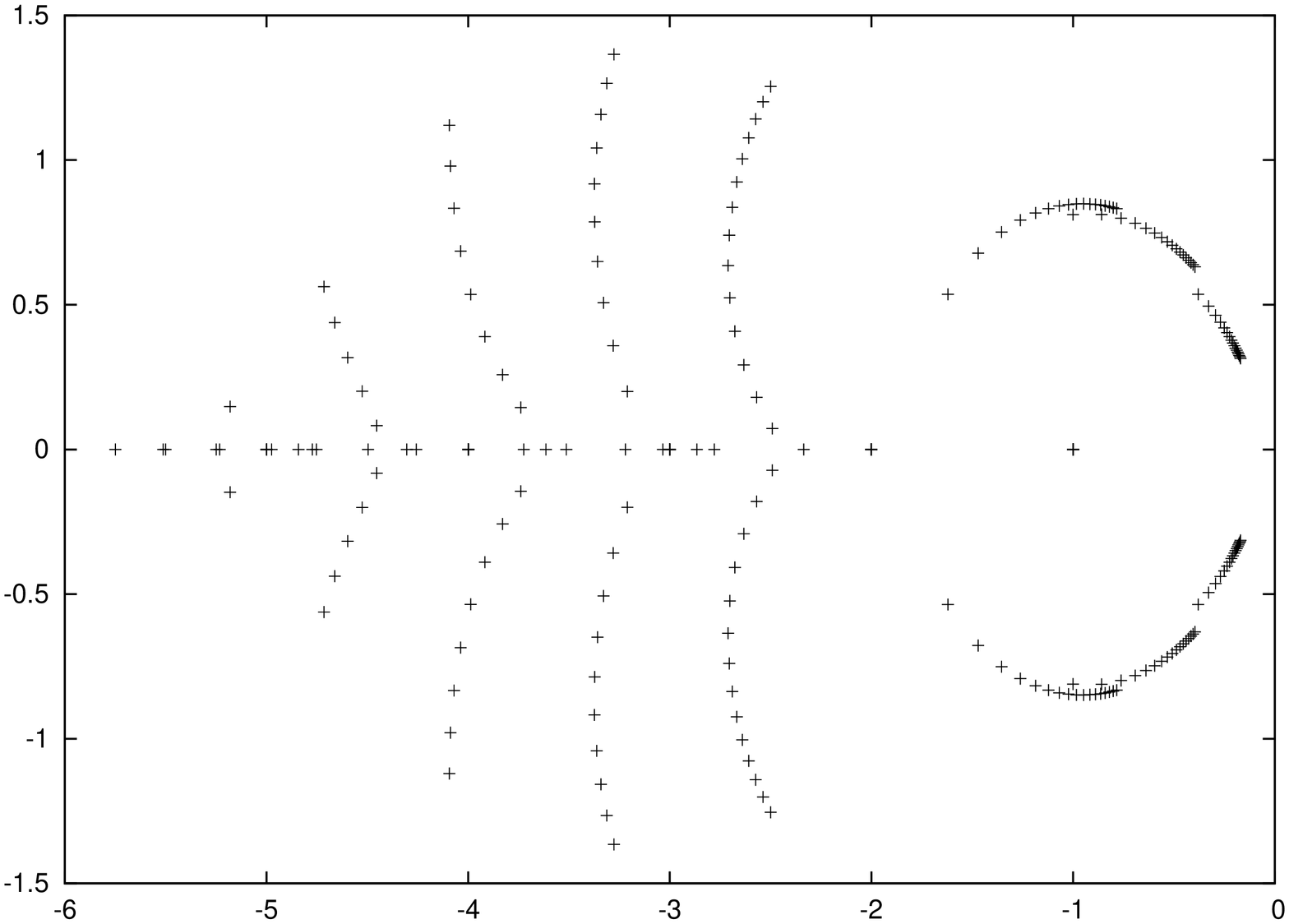}}
\caption{$d=4$}
\label{d4}
\end{center}
\end{minipage}
\begin{minipage}{0.5\hsize}
\begin{center}
\rotatebox{-90}{\includegraphics[width=50mm]{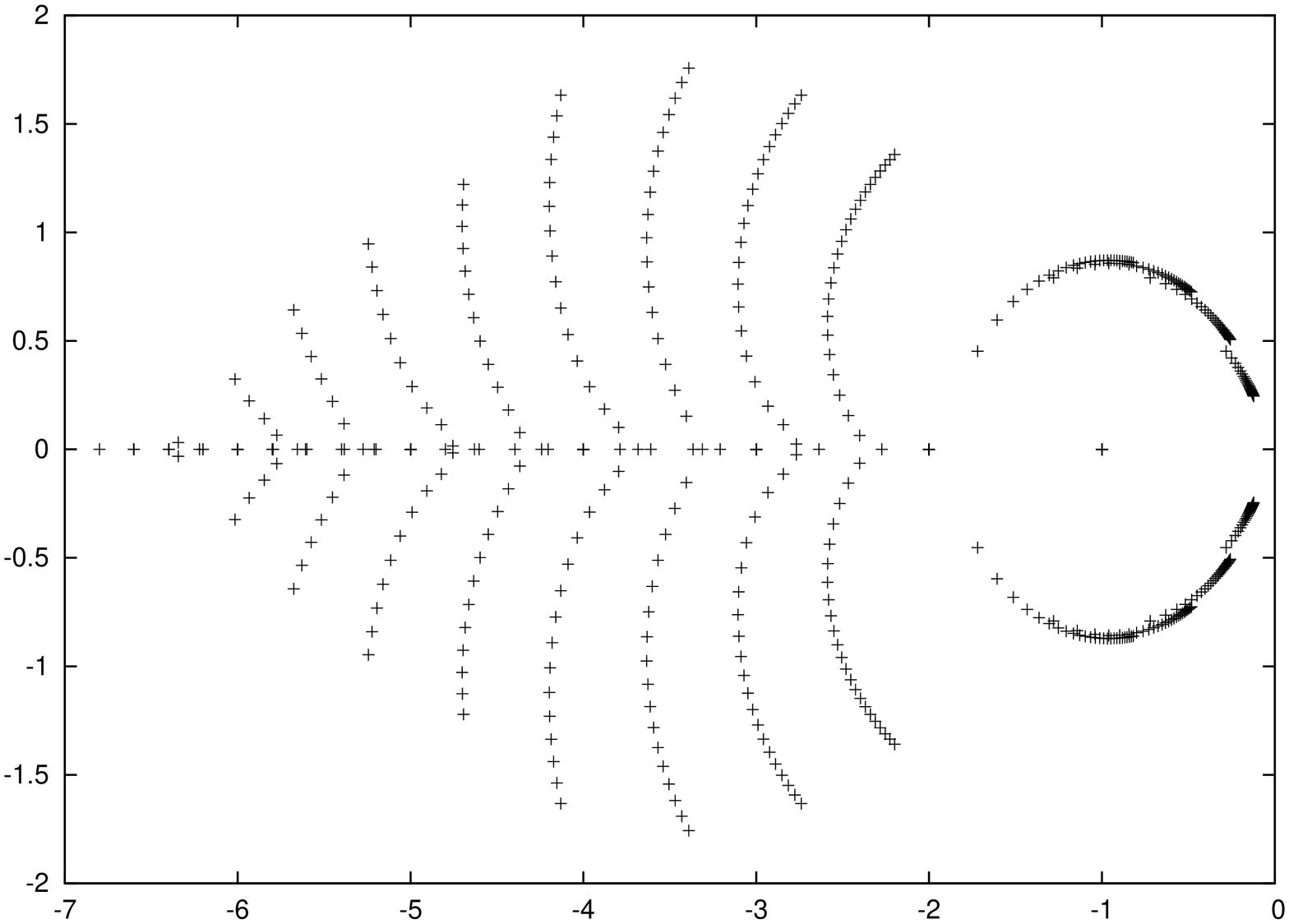}}
\caption{$d=5$}
\label{d5}
\end{center}
\end{minipage}
\end{tabular}
\end{center}
\end{figure}

\begin{figure}[htbp]
\begin{center}
\begin{tabular}{c}
\begin{minipage}{0.5\hsize}
\begin{center}
\rotatebox{-90}{\includegraphics[width=5cm]{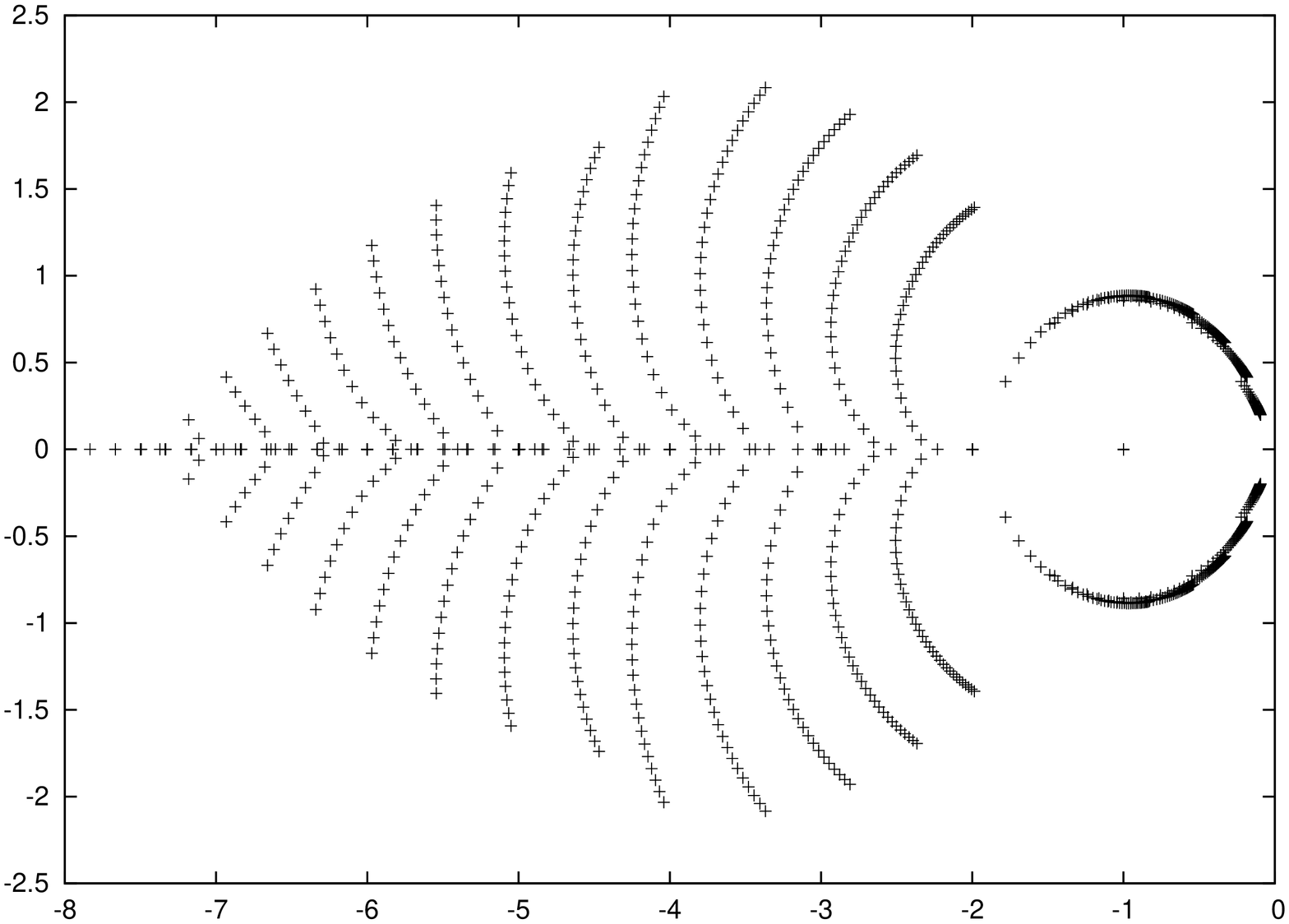}}
\caption{$d=6$}
\label{d6}
\end{center}
\end{minipage}
\begin{minipage}{0.5\hsize}
\begin{center}
\rotatebox{-90}{\includegraphics[width=50mm]{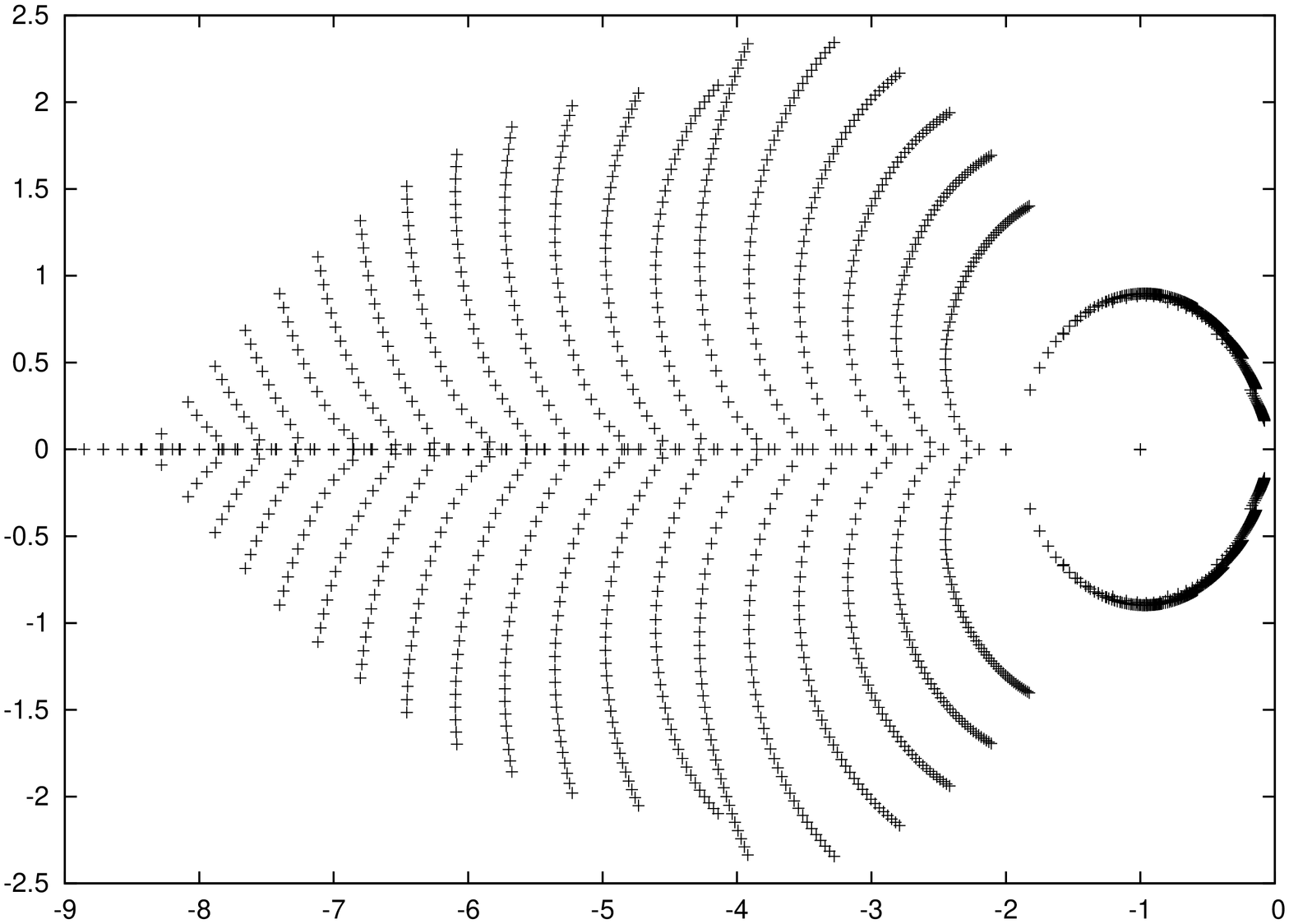}}
\caption{$d=7$}
\label{d7}
\end{center}
\end{minipage}
\end{tabular}
\end{center}
\end{figure}

\begin{figure}[htbp]
\begin{center}
\begin{tabular}{c}
\begin{minipage}{0.5\hsize}
\begin{center}
\rotatebox{-90}{\includegraphics[width=5cm]{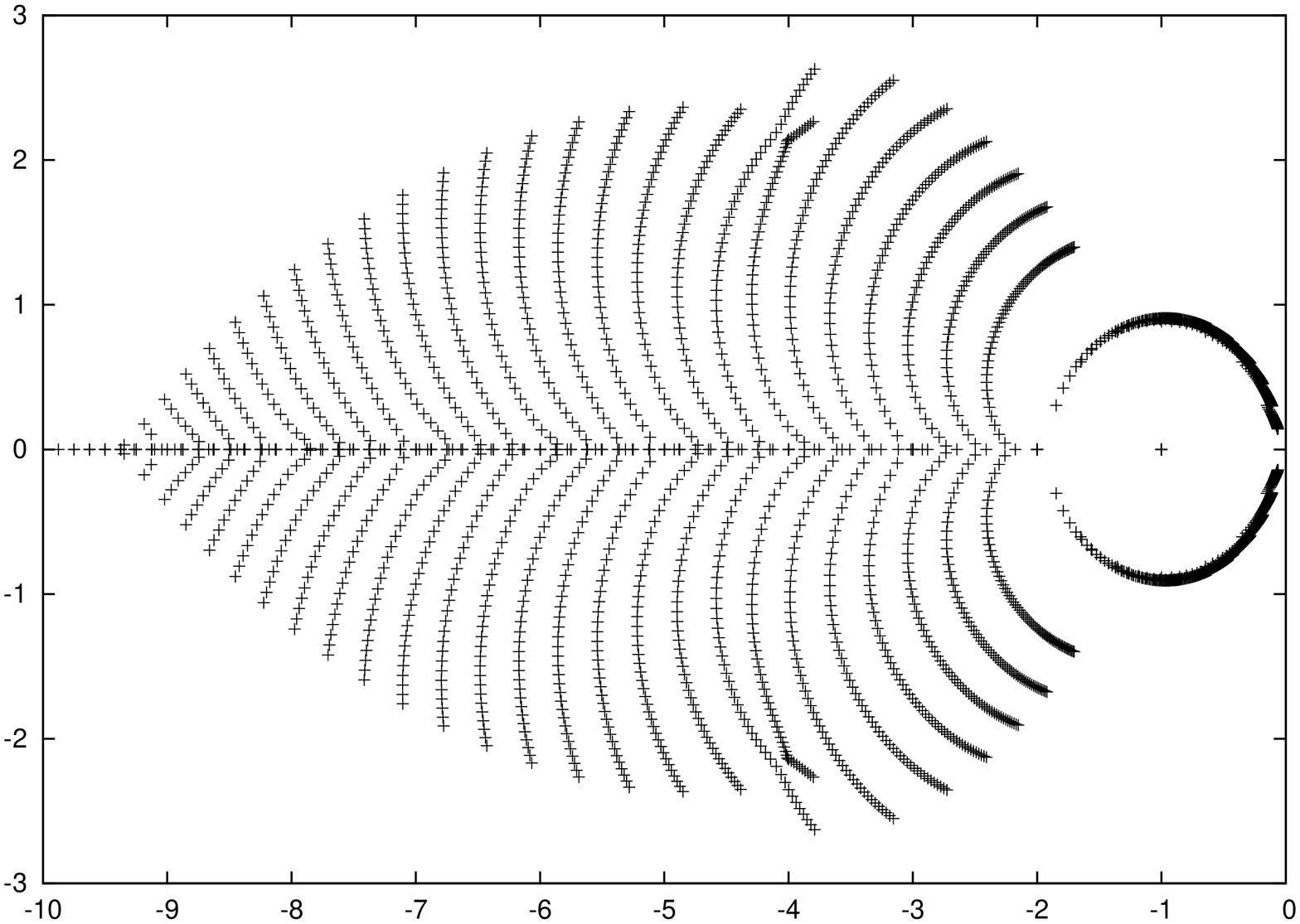}}
\caption{$d=8$}
\label{d8}
\end{center}
\end{minipage}
\begin{minipage}{0.5\hsize}
\begin{center}
\rotatebox{-90}{\includegraphics[width=50mm]{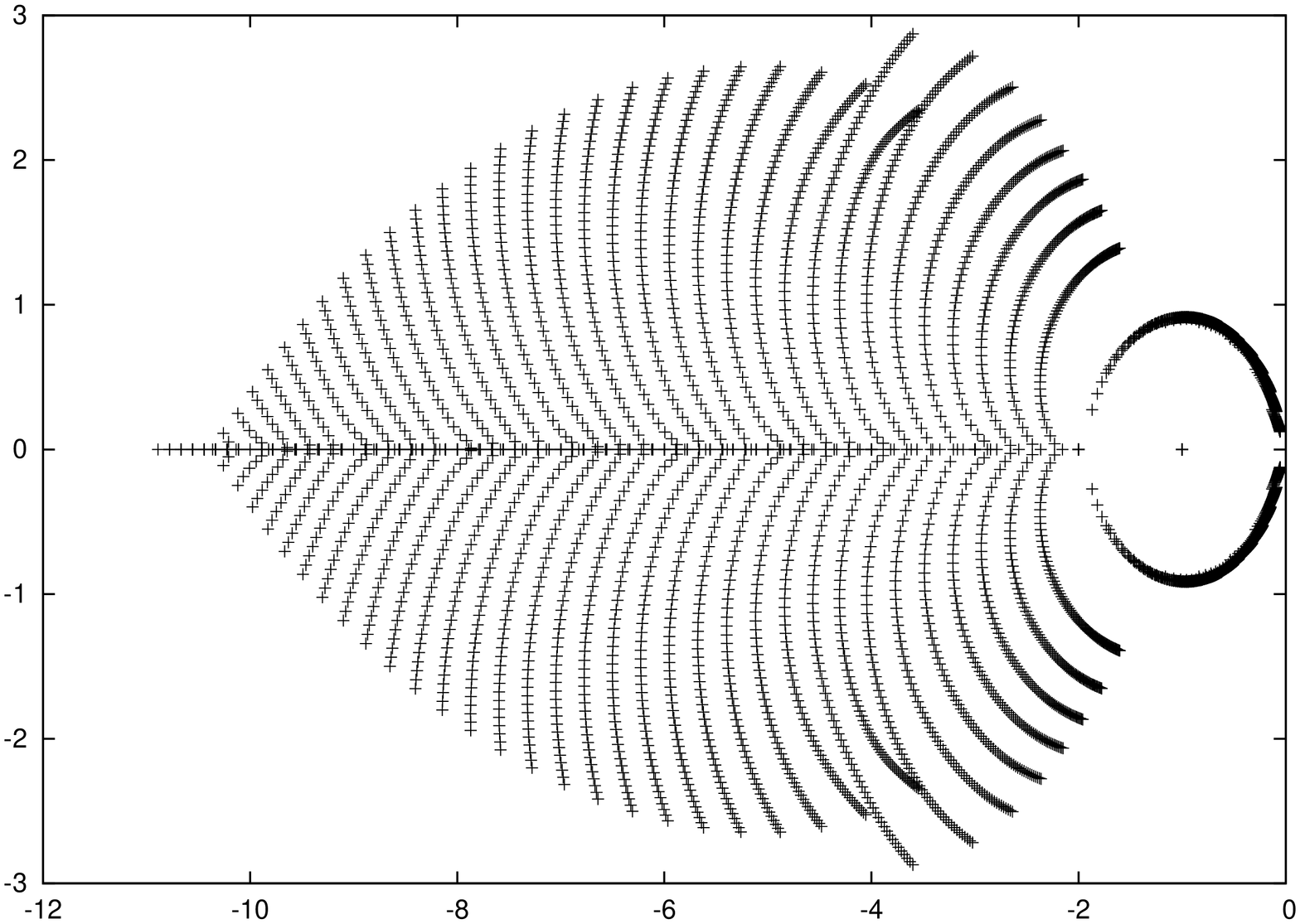}}
\caption{$d=9$}
\label{d9}
\end{center}
\end{minipage}
\end{tabular}
\end{center}
\end{figure}

\begin{figure}[htbp]
\begin{center}
\begin{tabular}{c}
\begin{minipage}{0.5\hsize}
\begin{center}
\rotatebox{-90}{\includegraphics[width=5cm]{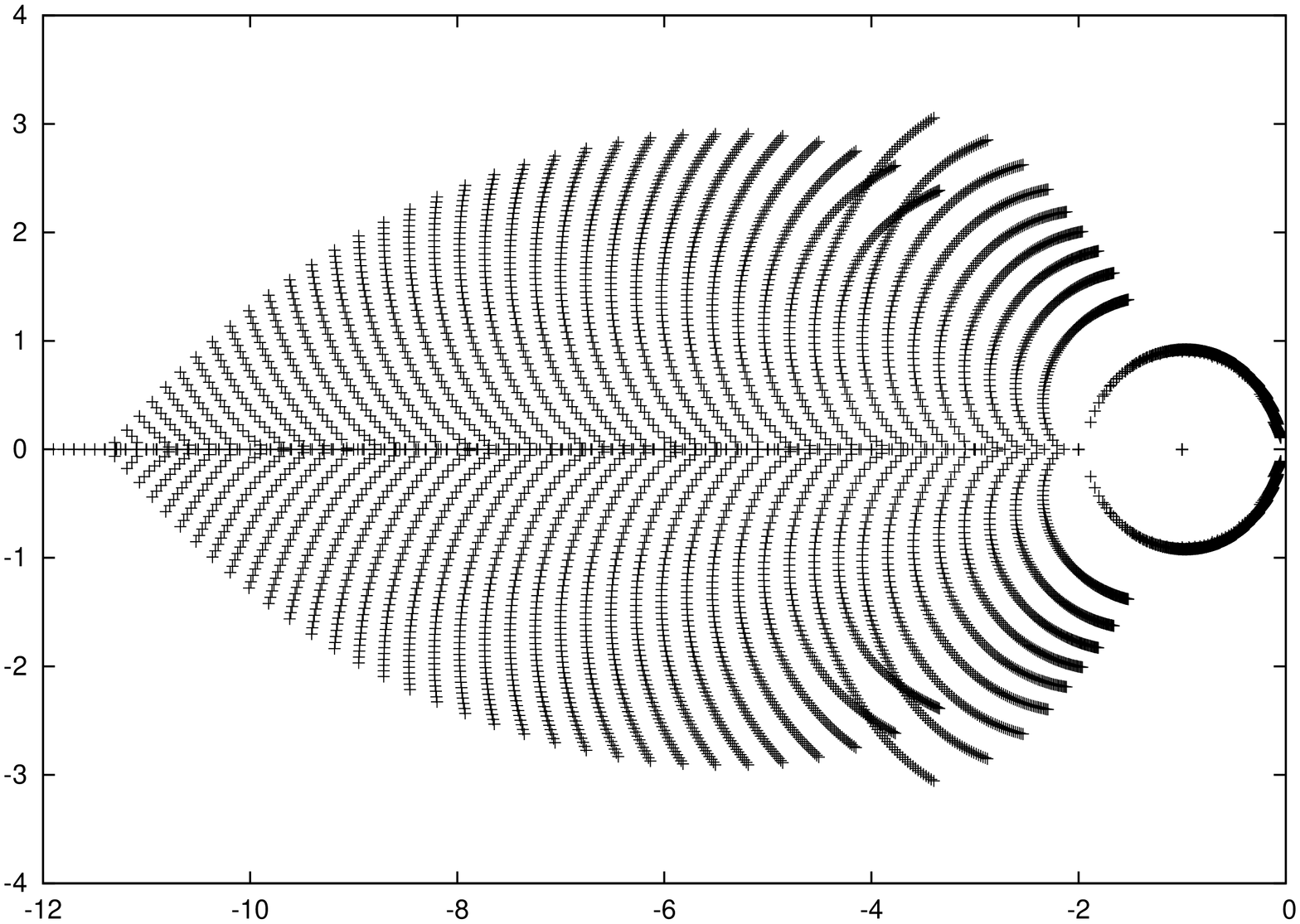}}
\caption{$d=10$}
\label{d10}
\end{center}
\end{minipage}
\begin{minipage}{0.5\hsize}
\begin{center}
\rotatebox{-90}{\includegraphics[width=50mm]{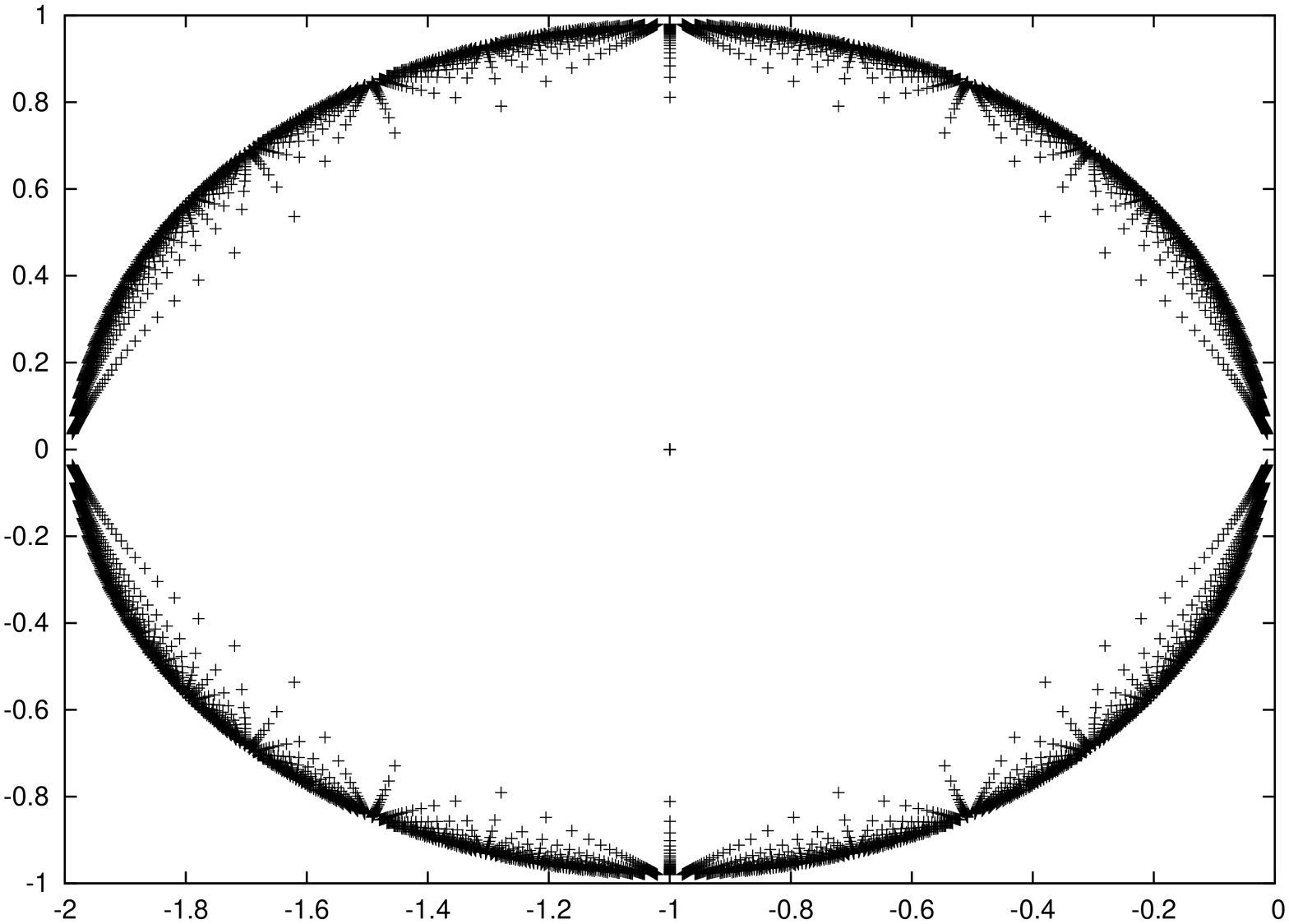}}
\caption{$4 \le d \le 75$ and $n=2d$}
\label{d4_75}
\end{center}
\end{minipage}
\end{tabular}
\end{center}

\end{figure}

\section*{Acknowledgement}

This research was supported by the JST CREST.

\end{document}